\documentclass[a4paper,11pt,reqno]{amsart} 
\usepackage{amsmath}
\usepackage{amsfonts}
\usepackage[T1]{fontenc}  
\usepackage[utf8]{inputenc}
\usepackage[mathscr]{euscript}
\usepackage{verbatim}
\usepackage{amssymb,latexsym}
\usepackage{amsthm}
\usepackage{color}
\usepackage{graphicx}
\usepackage[hidelinks]{hyperref}
\usepackage[font=scriptsize]{caption}
\usepackage{amsmath}
\usepackage{subcaption}
\usepackage{float}

\newcommand{\Ab}{\mathbf{A}}

\usepackage[lmargin=3.5 cm,rmargin=3.5 cm,tmargin=3.5cm,bmargin=2.5cm,paper=a4paper]{geometry}
\DeclareMathOperator{\curl}{curl}\DeclareMathOperator{\Div}{div}
  
\DeclareMathOperator{\supp}{supp} \DeclareMathOperator{\dom}{\mathrm {Dom}}

\newtheorem{thm}{Theorem}[section]

\newtheorem{theorem}[thm]{Theorem}

\newtheorem{lemma}[thm]{Lemma}
\newtheorem{proposition}[thm]{Proposition}

\theoremstyle{remark}

\newtheorem{rem}[thm]{Remark}

\newcommand{\nb}{\nabla}
\newcommand{\R}{\mathbb{R}}
\newcommand{\Fb}{\mathbf{F}}

\newcommand{\N}{\mathbb{N}}
\newcommand{\C}{\mathbb{C}}

\newcommand{\Om}{\Omega}
\newcommand{\kp}{\kappa}

\newcommand{\GL}{\mathcal E_{\kappa,H}}

\def\sig#1{\vbox{\hsize=5.5cm
		\kern2cm\hrule\kern1ex
		\hbox to \hsize{\strut\hfil #1 \hfil}}}
\newcommand\signatures[4]{%
	\vspace{3cm}
	\hbox to \hsize{\hfil #1, \today\hfil}
	\vspace{3cm}
	\hbox to \hsize{\quad#2\hfil\hfil #3\quad}
	\vspace{3cm}
	\hbox to \hsize{\hfil#4\hfil}}
\numberwithin{equation}{section}

\title[Lowest energy band function]{Lowest energy band function for magnetic steps}
\author[W. Assaad]{Wafaa Assaad}
\address{Lebanese University,  Doctoral School of Science and Technology, Laboratory of Mathematics,  Lebanon}
\email{wafaa\_assaad@hotmail.com}
\author[A. Kachmar]{Ayman Kachmar}
\address{Lebanese University, Department of Mathematics, Hadat, Lebanon}
\email{ayman.kashmar@gmail.com}
\begin{document}
\maketitle
\begin{abstract}
We study the Schr\"odinger operator in the plane with a step magnetic field function. The bottom of its spectrum is  described by the infimum of the lowest eigenvalue band function, for which we establish the existence and uniqueness of the non-degenerate minimum. We discuss  the curvature effects on the localization properties of magnetic ground states, among other applications.
\end{abstract}

\section{Introduction}\label{sec:int}
\subsection{The planar magnetic step operator}\label{sec:step2}
Let  $a \in [-1,1)\setminus\{0\}$. 
We define the self-adjoint magnetic Schr\"odinger operator on the plane 
\begin{equation}\label{eq:ham_operator}
\mathcal L_a=\partial^2_{x_2}+\big(\partial_{x_1}+i\sigma x_2\big)^2, \quad (x_1,x_2) \in \R^2,
\end{equation}
where $\sigma$ is a step function defined as follows
\begin{equation}\label{eq:sigma1}
\sigma(x_1,x_2)=\mathbf {1}_{\R_+}(x_2)+a\mathbf {1}_{\R_-}(x_2).
\end{equation}

The operator $\mathcal L_a$  is invariant w.r.t. translations in the $x_1$-direction, then it can be fibered and reduced to a family of 1D Shr\"odinger operators on $L^2(\R)$, $\mathfrak h_a[\xi]$, after a  Fourier transform along the $x_1$-axis
(see~\cite{hislop2016band,reed}). The fiber operators $\mathfrak h_a[\xi]$, parametrized by $\xi\in\R$, are defined in Section~\ref{sec:step1}.

We have the following link between the spectra of the operators $\mathcal L_a$ and $\mathfrak h_a[\xi]$ (see~~\cite{hislop2016band} and~\cite[Section~4.3]{fournais2010spectral}):
\begin{equation}\label{eq:L_h}
\mathrm{sp}\big(\mathcal L_a\big)=\overline{\bigcup_{\xi \in \R} \mathrm{sp}\big(\mathfrak h_a[\xi] \big)}.
\end{equation}
Consequently, the bottom of the spectrum of $\mathcal L_a$, denoted by $\beta_a$, can be computed by  minimizing  the ground state energies of the fibered operators $\mathfrak h_a[\xi] $ (see \eqref{eq:beta} below).

\subsection{The lowest energy band function}\label{sec:step1}

Let  $a \in [-1,1)\backslash\{0\}$. For all $\xi\in\R$, 
we introduce the operator 
$$\mathfrak h_a[\xi]=-\frac{d^2}{d\tau^2}+V_a(\xi,\tau),$$
 with the  potential $V_a(\xi,\tau)=\big(\xi+\sigma(\tau)\tau\big)^2$, where
\begin{equation}\label{eq:potential}
 \sigma(\tau)=\mathbf{1}_{\R_+}(\tau)+a\mathbf{1}_{\R_-}(\tau).
\end{equation}
The domain of $\mathfrak h_a[\xi]$ is given by~:
$${\bf Dom}\big(\mathfrak h_a[\xi]\big)=\left\{u\in B^1(\R)~:~\Big(-\frac{d^2}{d\tau^2}+V_a(\xi,\tau)\Big)u \in L^2(\R)\right\},$$
where  the space $B^n(I)$  is  defined for a positive integer $n$ and an open interval $I\subset \R$ as follows
\begin{equation} \label{eq:B_n}
B^n(I)=\{u \in L^2(I):\,\tau^i
\frac{d^ju}{d\tau^j} \in L^2(I),\,\forall i, j \in \N\ \mathrm{s.t.}\ i+j\leq n \}.
\end{equation} 
The quadratic form associated to $\mathfrak h_a[\xi]$ is
\begin{equation} \label{eq:quad}
q_a[\xi](u)=\int_{ \R} \big(|u'(\tau)|^2+V_a(\xi,\tau)|u(t)|^2 \big)\,d\tau
\end{equation}
defined on $B^1(\R)$.
The operator $\mathfrak h_a[\xi]$ is with compact resolvent. 
We introduce the lowest eigenvalue of this operator (lowest band function)
\begin{equation}\label{mu_a_1}
\mu_a(\xi)=\inf_{u\in B^1(\R),u\neq0} \frac{q_a[\xi](u)}{\|u\|^2_{L^2(\R)}}.
\end{equation}
This is a simple eigenvalue, to which corresponds a unique positive $L^2$-normalized eigenfunction, $\varphi_{a,\xi}$, i.e. satisfying (see~\cite[Proposition~A.2]{Assaad2019}).
\begin{equation}\label{eq:phi-a-p}
\varphi_{a,\xi}>0,~(\mathfrak h_a[\xi]-\mu_a[\xi])\varphi_{a,\xi}=0~\&~\int_{\R}|\varphi_{a,\xi}(\tau)|^2\,d\tau=1. 
\end{equation}
Moreover, the above eigenvalue and eigenfunction depend smoothly on $\xi$ (see \cite{DaHe,hislop2016band}),
\begin{equation}\label{eq:anal_a}
\xi\mapsto \mu_a(\xi)\ \mathrm{and}\ \xi\mapsto \varphi_{a,\xi}\ \mbox{are in}\
C^\infty.
\end{equation}
We introduce the \emph{step constant} (at $a$) as follows
\begin{equation}\label{eq:beta}
\beta_a:=\inf_{\xi \in \R} \mu_a(\xi),
\end{equation}
along with the celebrated de\,Gennes constant
\begin{equation}\label{eq:deG}
\Theta_0:=\beta_{-1}.
\end{equation}
Our main result is the following.

\begin{thm}\label{thm:main}
Given $a\in(-1,0)$, there exists a unique  $\zeta_a\in\R$ such that
	\[\beta_a=\mu_a(\zeta_a).\]
Furthermore,  the following holds.
\begin{enumerate}
\item $\zeta_a<0$ and satisfies $\mu_a''(\zeta_a)>0 $.
	\item $|a|\Theta_0<\beta_a<\Theta_0$.
	\item The ground state   $\phi_a:=\varphi_{a,\zeta_a}$ satisfies $\phi'_{a}(0)<0$.
\end{enumerate}
\end{thm}
\begin{rem}\label{rem:main}~
\begin{enumerate}
\item The existence of the minimum $\zeta_a$ was known earlier \cite{ Assaad2019,hislop2016band}. Our contribution establishes the uniqueness of $\zeta_a$ and that it is a non-degenerate minimum. These new properties were only conjectured in \cite{hislop2016band} based on numerical computations.
\item The case $a=-1$ is perfectly understood and can be reduced to the study of the de\,Gennes model (family of harmonic oscillators on the half-axis with Neumann condition at the origin). In this case, we know the existence of the unique and non-degenerate minimum $\zeta_{-1}=-\sqrt{\Theta_0}$, and that the ground state $\phi_{-1}$ is an even function with a vanishing derivative at the origin $(\phi_{-1}'(0)=0$).
\item Our comparison result $\beta_a<\Theta_0$ is also new. It was conjectured in \cite{Assaad2019} based on numerical computations\footnote{Many thanks to  V. Bonnaillie-No\"el for the numerical computations and Fig.~5 in \cite{Assaad2019}.}. This comparison has an interesting application to the existence of superconducting magnetic edge states (see Section~\ref{sec:GL}).
\item The sign of $\phi_a'(0)$ has an important application too, namely in precising the localization properties of ground states for the 
Schr\"odinger operator with magnetic steps and in the large field asymptotics. That will be  discussed in Section~\ref{sec:2D-curv}.
\item In the case $a\in(0,1)$, we have $\beta_a=a$ and $\mu_a(\cdot)$ does  not achieve a minimum. 
\end{enumerate}
\end{rem}

\section{The Robin model on the half line}\label{sec:robin}

We discuss in this section a model operator introduced in \cite{K-jmp, K-rmp}.
Let $\xi$ and $\gamma$ be two real parameters. We introduce the family of harmonic oscillators on $\R_+$,
\begin{equation}\label{eq:H_gamma}
H[\gamma,\xi]=-\frac{d^2}{d\tau^2}+(\tau+\xi)^2,
\end{equation}
with the following operator domain (accommodating functions  satisfying the Robin condition at the origin) 
\begin{equation}\label{eq:domRob}
{\bf Dom}\big(H[\gamma,\xi] \big)=\{u \in B^2(\R_+)~:~ u'(0)=\gamma u(0)\}.
\end{equation}
The quadratic form associated to $H[\gamma,\xi]$ is
$$B^1(\R_+)\ni u\longmapsto q[\gamma,\xi](u)=\int_{\R_+} \Big(|u'(\tau)|^2+|(\tau+\xi)u(\tau)|^2 \Big)\,d\tau +\gamma|u(0)|^2.$$
The operator $H[\gamma,\xi]$ is with compact resolvent, hence its spectrum  is an increasing sequence of eigenvalues $\lambda^j(\gamma,\xi)$, $j\in\N^*$. Furthermore, these eigenvalues are simple (see~\cite[Section~3.2.1]{fournais2010spectral} for the argument). Consequently, we introduce the corresponding orthonormal family of eigenfunctions  $u^j_{\gamma,\xi}$ satisfying
\begin{equation}\label{eq:u-gj>0}
u^j_{\gamma,\xi}(0)>0.
\end{equation}
The condition in \eqref{eq:u-gj>0} determines the \emph{normalized} eigenfunction uniquely, because $u^j_{\gamma,\xi}(0)\not=0$,  otherwise it will vanish everywhere by Cauchy's uniqueness theorem,  since $(u^j_{\gamma,\xi})'(0)=\gamma u^j_{\gamma,\xi}(0)$ and
\[ -\frac{d^2}{d\tau^2}u^j_{\gamma,\xi}+(\tau+\xi)^2u^j_{\gamma,\xi}=\lambda^j(\gamma,\xi)u^j_{\gamma,\xi} ~{\rm on~}\R_+.\]
The perturbation theory  ensures that the functions
 \begin{equation}\label{eq:lambda-anal}
\xi\mapsto \lambda^j(\gamma,\xi),\ \xi\mapsto u^j_{\gamma,\xi},\
\gamma\mapsto \lambda^j(\gamma,\xi),\ \mathrm{and}\ \gamma\mapsto u^j_{\gamma,\xi}\ \mathrm{are}\  C^\infty.
 \end{equation} 
The reader is referred to~\cite{Kato} (for general perturbation theory) and~\cite[Theorem C.2.2]{fournais2010spectral}) for the application in the present context.
 
The first partial derivatives of the eigenvalues with respect to $\xi$ and $\gamma$ are as follows (see \cite{K-jmp, K-rmp})
\begin{align}
\partial_\xi \lambda^j(\gamma,\xi)&=\big(\lambda^j(\gamma,\xi)-\xi^2+\gamma^2\big)|u^j_{\gamma,\xi}(0)|^2,\label{eq:der_gamma}\\
\partial_\gamma \lambda^j(\gamma,\xi)&=|u^j_{\gamma,\xi}(0)|^2.\label{eq:der_gamma1}
\end{align}
Using the min-max principle, the lowest eigenvalue is defined  as follows:
\begin{equation}\label{eq:mu1}
\lambda(\gamma,\xi):=\lambda^1(\gamma,\xi)=\inf {\rm sp}\big(H[\gamma,\xi]\big)=\inf_{u\in B^1(\R),u\neq0} \frac{q[\gamma,\xi](u)}{\|u\|^2_{L^2(\R)}}.
\end{equation}
Note that the \emph{normalized}  ground state,  $u_{\gamma,\xi}$, does not change  sign on $\R_+$, and hence it is positive by our choice in \eqref{eq:u-gj>0}.

For $\gamma\in \R$, we introduce the \emph{de\,Gennes function},
\begin{equation}\label{eq:theta_gamma}
\Theta(\gamma):=\inf_{\xi \in \R} \lambda(\gamma,\xi).
\end{equation}
\begin{theorem}\label{thm:sturm} \emph{(\cite{DaHe, K-jmp})}~

The following statements hold
	\begin{enumerate}
		\item For all $\xi\in\R$, $\gamma\mapsto \lambda(\gamma,\xi)$ is increasing.
		\item For all $\gamma\in\R$,  $\displaystyle\lim_{\xi \rightarrow -\infty}\lambda(\gamma,\xi)=1$ and $\displaystyle\lim_{\xi \rightarrow +\infty}\lambda(\gamma,\xi)=+\infty$.
		\item  For all $\gamma\in\R$, the function $\xi \mapsto \lambda(\gamma,\xi)$ admits a unique minimum attained at 	\begin{equation}\label{eq:psi-gamma}
		\xi(\gamma):=-\sqrt{\Theta(\gamma)+\gamma^2}.
		\end{equation}
		Furthermore, this minimum is non-degenerate, $\partial^2_\xi\lambda(\gamma,\xi(\gamma))>0$.
		\item For all $\gamma\in\R$, $-\gamma^2\leq\Theta(\gamma)<1$.
	\end{enumerate}
\end{theorem}

{\it The Neumann realization.} The particular case where $\gamma=0$ corresponds to the Neumann realization of the operator $H[0,\xi]$, denoted by $H^N[\xi]$, with the associated quadratic form $q^N[\xi]=q[0,\xi]$. The first eigenvalue of $H^N[\xi]$ is denoted by
\begin{equation}\label{eq:mu_n}
\lambda^N(\xi)=\inf {\rm sp}\big(H^N[\xi]\big)=\lambda(0,\xi),
\end{equation}
with the corresponding positive $L^2$-normalized eigenfunction $u^N_\xi:=u_{0,\xi}$. 

By a symmetry argument \cite{Assaad2019,hislop2016band}, we get that the \emph{step} constant $\beta_{-1}$ (in~\eqref{eq:beta}) satisfies
\begin{equation}\label{eq:theta_0}
\Theta_0:=\beta_{-1}=\Theta(0).
\end{equation}
This universal value $\Theta_0$ is often named the \emph{de\,Gennes} constant in the literature \cite{FHP, fournais2010spectral} and satisfies $\Theta_0\in(\frac12,1)$. Numerically (see~\cite{Bon}), one finds $\Theta_0\sim 0.59$.  Note that the non-degenerate minimum $\xi_0:=\xi(0)$ of $\mu^N(\cdot)$ satisfies $\xi_0=-\sqrt{\Theta_0}$.
\section{The step model on the line}\label{sec:ha}
We analyze the band function $\mu_a(\cdot)$ introduced in \eqref{mu_a_1} along with the \emph{positive} normalized ground state $\varphi_{a,\xi}$. 

Note that we are focusing on the interesting situation where $a\in(-1,0)$. As mentioned earlier, for $a\in(0,1)$, the minimum of $\mu_a(\cdot)$ is not achieved and the step constant $\beta_a=a$ \cite{Assaad2019,hislop2016band}; while for $a=-1$,  the case reduces to the de\,Gennes model and $\beta_{-1}=\Theta_0$.

\subsection{Preliminaries}

Left with the situation $a\in(-1,0)$, it is known that a minimum $\zeta_a$ exists and must be negative, $\zeta_a<0$ \cite[Prop.~A.7]{Assaad2019}; our Theorem~\ref{thm:main} sharpens this by establishing that the minimum is unique and non-degenerate. To prove this, new comparison estimates of the step constant $\beta_a$ are needed which improve the existing estimates in the literature \cite{ Assaad2019,hislop2016band}.

The existence of a minimum is due to the behavior at infinity of the band function $\mu_{a}(\cdot)$,   namely,
\begin{equation*}\label{eq:lim_a}
\lim_{\xi\rightarrow-\infty}\mu_a(\xi)= |a|  \quad {\rm and}\quad \lim_{\xi\rightarrow+\infty}\mu_a(\xi)=+\infty,
\end{equation*}
and  the following estimates on the step constant,
\begin{equation}\label{eq:theta4}
|a|\Theta_0<\beta_a<|a|.
\end{equation}
Note that the lower bound  \eqref{eq:theta4} results from a simple comparison arguments using the min-max principle (see \cite[Prop.~A.6]{Assaad2019}); the upper bound is more tricky and relies on the construction of a trial state related to the Robin model introduced in Section~\ref{sec:robin} (see e.g. \cite[Thm.~2.6]{Assaad2019}).
Finally, we recall the expression for the derivative of $\mu_a(\cdot)$ established in \cite{hislop2015edge} (see also~\cite[Prop. A.4]{Assaad2019}). 
\begin{equation}\label{eq:mu_deriv}
\mu_a'(\xi)=\Big(1-\frac1a\Big)\Big ({\varphi'_{a,\xi}(0)}^2+\big(\mu_a(\xi)-\xi^2\big){\varphi_{a,\xi}(0)}^2\Big).
\end{equation}

\subsection{Comparison with the de\,Gennes constant}

\begin{proposition}\label{prop:beta-theta}
	Let $a\in(-1,0)$. For $\beta_a$ and $\Theta_0$ as in~\eqref{eq:beta} and~\eqref{eq:theta_0} respectively, we have
	\[\beta_a<\Theta_0.\]
\end{proposition}
\begin{proof}
If $a\in[-\Theta_0,0)$, then \eqref{eq:theta4} yields that $\beta_a<\Theta_0$ and the conclusion of 	Proposition~\ref{prop:beta-theta} follows in this particular case.

In the sequel, we fix $a\in(-1,\Theta_0)$. For all $\xi\in\R$, we denote by $u(\cdot;\xi)=u^N_\xi(\cdot)$ the positive ground state of the de\,Gennes model (corresponding to the eigenvalue $\lambda^N(\xi)$ in \eqref{eq:mu_n}. We introduce the  function $g_\xi$ on $\R$ as follows:
	\begin{equation}\label{eq:g}
	g_\xi(\tau)= 
	\begin{cases}
u(\tau;\xi),&\mathrm{if}~t\geq 0,\\
	c u(\tau;\xi/\sqrt{|a|}),&\mathrm{if}~t<0,
	\end{cases}
	\end{equation}
	with $c=c_\xi:=u(0;\xi)/u(0; \xi/\sqrt{|a|})>0$  so that $g_\xi(0^-)=g_\xi(0^+)$. We observe that $g_\xi$ is in the form domain of the operator $\mathfrak h_a[\xi]$. 
	Performing an elementary scaling argument, we get
	\begin{align*}
q_a[\xi](g_\xi)&= \lambda^N(\xi)\int_{\R_+}|g_\xi(t)|^2\,dt+|a|\lambda^N\Big(\frac {\xi}{\sqrt{|a|}}\Big)\int_{\R_-}|g_\xi(t)|^2\,dt\\
&=\lambda^N(\xi)\int_{\R}|g_\xi(t)|^2\,dt+\left(|a|\lambda^N\Big(\frac {\xi}{\sqrt{|a|}}\Big)-\lambda^N(\xi)\right)\int_{\R_-}|g_\xi(t)|^2\,dt.
	\end{align*}
We choose now $\xi=\xi_0:=-\sqrt{\Theta_0}$ corresponding to $\Theta_0$ in~\eqref{eq:theta_0}. That way, we get $\lambda^N(\xi_0)=\Theta_0$ and
\[q_a[\xi_0](g_{\xi_0})=\Theta_0\int_{\R}|g_{\xi_0}(\tau)|^2\,dt+f(|a|)\int_{\R_-}|g_{\xi_0}(\tau)|^2\,d\tau,\]
where $f(x):=x\lambda^N\big(\frac {\xi_0}{\sqrt{x}}\big)-\Theta_0$,	for $x\in(\Theta_0,1)$. By the min-max principle
\[\beta_a\leq \frac {q_a[\xi_0](g_{\xi_0})}{\|g_{\xi_0}\|^2_{L^2{(\R)}}}\leq \Theta_0+f(|a|)\frac {\int_{\R_-}|g_{\xi_0}(\tau)|^2\,d\tau}{\int_{\R}|g_{\xi_0}(\tau)|^2\,d\tau}.\]
To get that $\beta_a<\Theta_0$,  it suffices to prove that $f(x)<0$, for $x\in(\Theta_0,1)$.

Let $x\in(\Theta_0,1)$ and $\alpha=\frac{\xi_0}{\sqrt{x}}\in(-1,\xi_0)$. By \eqref{eq:der_gamma} (applied for $j=1$ and $\gamma=0$), we can write
\[
f(x)
=x\big(\lambda^N(\alpha)-\alpha^2\big)=
x\frac{(\lambda^N)'(\alpha)}{|u^N_\alpha(0)|^2}.
\]
Since $\alpha\in(-1,\xi_0)$ and $\lambda^N(\cdot)$ is monotone decreasing on the interval $(-1,\xi_0)$, we deduce that $(\lambda^N)'(\alpha)<0$ and eventually $f(x)<0$ as required.
\end{proof}

\subsection{Variation of the ground state near zero}

We pick any $\zeta_a\in\mu_a^{-1}(\beta_a)$ so that $\beta_a=\mu_a(\zeta_a)$, and denote by $\phi_a=\varphi_{a,\zeta_a}$ the positive normalized ground state for $\beta_a$ (so we are suppressing the dependence of the ground state  on $\zeta_a$). We determine the sign of the derivative of $\phi_a$ at the origin, thereby yielding that the ground state is a decreasing function in a neighborhood of $0$. This result will be crucial in deriving the sign of some \emph{moments} in Section~\ref{sec:moment} later.
\begin{proposition}\label{prop:sign-der}
For all $a\in (-1,0)$ and $\zeta_a\in\mu_a^{-1}(\beta_a)$,  the positive normalized ground state $\phi_a=\varphi_{a,\zeta_a}$ satisfies $\phi'_a(0)<0$.
\end{proposition}
\begin{proof}
The proof relies on a comparison procedure with the Robin model.  Let $\gamma_a=\phi_a'(0)/\phi_a(0)$. Since the ground state $\phi_a$ is positive,  it suffices to prove that $\gamma_a<0$. The eigenvalue equation $\mathfrak h_a[\zeta_a]\phi_a=\beta_a\phi_a$ written on $\R_+$ is
 \begin{equation}\label{eq:ode1} 
 \begin{cases}
-\phi_a''(\tau)+(\tau+\zeta_a)^2\phi_a(\tau)=\beta_a \phi_a(\tau),&\quad t> 0,\\
\phi_a'(0)=\gamma_a \phi_a(0),&
 \end{cases}
 \end{equation}
Consequently, $\phi_a$ is an eigenfunction of the Robin operator $H[\gamma_a,\zeta_a]$, defined in~\eqref{eq:domRob}, with a corresponding eigenvalue $\beta_a$. Using the min-max principle, we have
 \begin{equation}\label{eq:beta-lamda}
\beta_a\geq
\lambda(\gamma_a,\zeta_a)
 \end{equation}
  where $\lambda(\gamma_a,\zeta_a)$ is defined in~\eqref{eq:mu1}. 
   
  If $\gamma_a\geq0$, then by~Theorem~~\ref{thm:sturm}, Proposition~\ref{prop:beta-theta} and~\eqref{eq:theta_0}, we get
  \[\lambda(\gamma_a,\zeta_a)\geq \lambda(0,\zeta_a)=\lambda^N(\zeta_a)\geq\Theta_0>\beta_a,\]
  thereby contradicting~\eqref{eq:beta-lamda}. This proves that $\gamma_a<0$.  
\end{proof}

\subsection{Uniqueness and non-degeneracy of the minimum}\label{sec:uniq}

Now, we establish that the minimum of $\mu_a(\cdot)$ is unique and non-degenerate. The key in our proof is a tricky connection with the Robin model. 
\begin{proposition}\label{prop:uniq}
	For all  $a\in(-1,0)$, 
\[\exists\,\zeta_a<0,~\mu_a^{-1}(\beta_a)=\{\zeta_a\}~\&~\mu_a''(\zeta_a)>0,\]
where $\mu_a(\cdot)$ and  $\beta_a$ are the eigenvalues introduced in \eqref{mu_a_1} and ~\eqref{eq:beta} respectively. 
\end{proposition}
\begin{proof}
	First, note that $\mu_a^{-1}(\beta_a)\subset\R_-$ and is non-empty, by \cite[Proposition~A.7]{Assaad2019}. Hence, it suffices to prove that any negative critical point  is a non-degenerate local minimum.
	
	Let $\eta<0$ be a critical point of $\mu_a(\cdot)$ (i.e. $\mu_a'(\eta)=0$). 
	For all $\xi\in\R$, we introduce 
	\begin{equation}\label{eq:gm}
	\gamma(\xi)=\gamma_a(\xi):=\varphi_{\xi,a}'(0)/\varphi_{\xi,a}(0),
	\end{equation}
	where $\varphi_{\xi,a}$ is the \emph{positive} normalized ground state of  the operator $\mathfrak h_a[\xi]$, which is now an eigenfunction for the Robin problem
	\begin{equation}\label{eq:ode2} 
	\begin{cases}
	-\varphi_{\xi,a}''(\tau)+(\tau+\xi)^2\varphi_{\xi,a}(\tau)=\mu_a(\xi) \varphi_{\xi,a}(\tau),&\quad \tau> 0,\\
	\varphi_{\xi,a}'(0)=\gamma(\xi) \varphi_{\xi,a}(0).&
	\end{cases}
	\end{equation}
Using this for $\xi=\eta$, we can pick $j=j(\eta)\in\mathbb N$ such that $\mu_a(\eta)=\lambda^j(\gamma(\eta),\eta)$,  the $j$th min-max eigenvalue of $H[\gamma(\xi),\xi]$. By the continuity of the involved functions and the simplicity of the eigenvalue $\lambda^j(\gamma(\eta),\eta)$, we can pick $\epsilon=\epsilon(\eta)>0$ such that
\begin{equation}\label{eq:R-ev}
\mbox{for all}\ \xi\in(\eta-\epsilon,\eta+\epsilon),\ \mu_a(\xi)=\lambda^j(\gamma(\xi),\xi).
\end{equation}
Hence, by~\eqref{eq:der_gamma}, \eqref{eq:der_gamma1} and differentiation in \eqref{eq:R-ev} w.r.t. $\xi$ we get
	\begin{align}
	\mu_a'(\xi)&=\partial_\xi \lambda^j(\gamma(\xi),\xi)\nonumber\\
	&=\big(\lambda^j(\gamma(\xi),\xi)-\xi^2+\gamma^2(\xi)\big)|u^j_{\gamma(\xi),\xi}(0)|^2+\gamma'(\xi) |u^j_{\gamma(\xi),\xi}(0)|^2.\label{eq:der-mu}
	\end{align}
	 Since $ \mu_a'(\eta)=0$, we infer from \eqref{eq:mu_deriv} and \eqref{eq:R-ev} that
\begin{equation}\label{eq:der-eta} \lambda^j(\gamma(\eta),\eta)-\eta^2+\gamma(\eta)^2= \mu_a(\eta)-\eta^2+\gamma(\eta)^2 =\frac{\mu_a'(\eta)}{\varphi_{\eta,a}(0)^2}=0.  
\end{equation}
	 Inserting this into \eqref{eq:der-mu} after setting $\xi=\eta$, we get  (thanks to \eqref{eq:u-gj>0})
	\begin{equation}\label{eq:gamma-der} 
	 \gamma'(\eta)=0.
	\end{equation}
	This result will be used in the computation of $\mu_a''(\eta)$ below. In fact,
	differentiation in~\eqref{eq:mu_deriv} w.r.t. $\xi$ yields
	\[\mu_a''(\xi)=\Big(1-\frac 1a\Big )\Big(\big( \mu_a(\xi)-\xi^2+\gamma(\xi)^2\big)\partial_\xi\varphi^2_{\xi,a}(0)+\big( \mu'_a(\xi)-2\xi+2\gamma(\xi)\partial_\xi\gamma(\xi)\big)\varphi^2_{\xi,a}(0)\Big).
	\]
 Considering again $\xi=\eta$, we get
	\[\mu_a''(\eta)=2\Big(\frac 1a-1\Big)\eta\varphi^2_{\eta,a}(0).\]
	In the above equation, we used~\eqref{eq:mu_deriv},~\eqref{eq:der-eta} and~\eqref{eq:gamma-der}.
	Recall that we take $\eta<0$ and $a\in(-1,0)$, hence 
	\[\mu_a''(\eta)>0,\]
	and this holds for any negative critical point, $\eta$, of $\mu_a(\cdot)$.  This finishes the proof.
\end{proof}

\subsection{Proof of the main result}
  Theorem~\ref{thm:main}
now follows by collecting Propositions~\ref{prop:uniq},~\ref{prop:sign-der} and \ref{prop:beta-theta}.

\section{Applications}\label{sec:app}

\subsection{Moments}\label{sec:moment}\
 
Fix  $a\in[-1,0)$ and consider $\beta_a$ as in \eqref{eq:beta}, the ground state $\phi_a$, and $\zeta_a$ the unique minimum of $\mu_a(\cdot)$ (see Theorem~\ref{thm:main} and Remark~\ref{rem:main}). We can invert the operator
$\mathfrak h_a[\zeta_a]-\beta_a$ on the functions orthogonal to the ground state $\phi_a$, thereby leading to the introduction of the regularized resolvent (see e.g.~\cite[Lemma~3.2.9]{fournais2010spectral}):
\begin{equation}\label{eq:R}
\mathfrak R_a(u)=
\begin{cases}
0&\mathrm{if}~u\parallel\phi_a\\
(\mathfrak h_a[\zeta_a]-\beta_a)^{-1}u &\mathrm{if}~u\perp \phi_a
\end{cases}.
\end{equation}
The construction of certain trial states in Sec.~\ref{sec:1Dw} below requires inverting  $\mathfrak h_a[\zeta_a]-\beta_a$ on functions involving $(\zeta_a+\sigma(\tau)\tau )^n\phi_a(\tau)$, for  positive integers $n$, with $\sigma(\cdot)$ introduced in~\eqref{eq:potential}. We are then lead to investigate the following \emph{moments}
\[M_n(a)=\int_{-\infty}^{+\infty}\frac 1{\sigma(\tau)}(\zeta_a+\sigma(\tau)\tau)^n|\phi_a(\tau)|^2\,d\tau,\]
\begin{proposition}\label{prop:moment}
For $a\in[-1,0)$, we have
	\begin{align}
	M_1(a)&=0,\label{eq:m1}\\
	M_2(a)&=-\frac 12 \beta_a\int_{-\infty}^{+\infty}\frac 1{\sigma(t)}|\phi_a(\tau)|^2\,d\tau+\frac 14\Big(\frac 1a-1\Big)\zeta_a\phi_a(0)\phi_a'(0), \label{eq:m2}\\
	M_3(a)&=\frac 13\Big(\frac 1a-1\Big)\zeta_a\phi_a(0)\phi_a'(0).\label{eq:m3}
	\end{align}
\end{proposition}
\begin{rem}\label{rem:moment}~
\begin{enumerate}
\item (Feynman-Hellmann)\
We have  (see e.g.  \cite[Eq.~(A.9)]{Assaad2019})
\begin{equation}\label{eq:FH-ms}
(\zeta_a+\sigma(\tau)\tau)\phi_a(\tau)\perp\phi_a(\tau)~{\rm  in~}L^2(\R).
\end{equation}  
Furthermore, since $M_1(a)=0$, we get further that $\frac1{\sigma(\tau)}(\zeta_a+\sigma(\tau)\tau)\phi_a\perp\phi_a$ too. Combined together, we see that
\[ (\zeta_a+a\tau)\phi_a\perp\phi_a~{\rm in~}L^2(\R_-)~\&~(\zeta_a+\tau)\phi_a\perp\phi_a~{\rm in~}L^2(\R_+)\]
which is consistent with \eqref{eq:ode1}, since by \eqref{eq:mu_deriv} and \eqref{eq:der_gamma}, $\zeta_a$ is a critical point of the corresponding Robin band function $\lambda^j(\gamma_a,\cdot)$. 
\item
As a consequence of Theorem~\ref{thm:main}, $M_3(a)=0$ for $a=-1$, and it is negative for $-1<a<0$, which  is consistent with \cite{BS98}.
\end{enumerate}
\end{rem}
\begin{proof}
	In an analogous manner to \cite{BS98}, we define the operator
	\[L=\mathfrak h_a[\zeta_a]-\beta_a=-\frac{d^2}{d\tau^2}+(\zeta_a+\sigma(\tau)\tau)^2-\beta_a.\]
	Pick an arbitrary smooth function on $\R\setminus\{0\}$ and set $v=2p\phi_a'-p'\phi_a$. We check that
	\begin{equation}\label{eq:m11}
Lv=\left(p^{(3)}-4\big((\zeta_a+\sigma \tau)^2-\beta_a\big) p'-4\sigma(\zeta_a+\sigma \tau)p\right)\phi_a.
	\end{equation}
	Noting that $L\phi_a=0$, we obtain by an integration by parts,
	\begin{align}\label{eq:m12}
\int_{-\infty}^{+\infty}\phi_aLv\,d\tau&=\int_{-\infty}^{+\infty}vL\phi_a\,d\tau-\phi_a(0)v'(0^-)+\phi_a(0)v'(0^+)+\phi'_a(0)v(0^-)-\phi_a'(0)v(0^+)\nonumber\\
&=-\phi_a(0)v'(0^-)+\phi_a(0)v'(0^+)+\phi'_a(0)v(0^-)-\phi_a'(0)v(0^+).
	\end{align}
 Take $p=1/\sigma^2$, then a simple computation, using~\eqref{eq:m11} and~\eqref{eq:m12}, yields
	\[M_1(a)=\frac 12\Big(1-\frac 1{a^2}\Big)\big((\beta_a-\zeta_a^2)\phi_a(0)^2+\phi_a'(0)^2\big).\]
	The definition of $\zeta_a$ ensures that $\mu'(\zeta_a)=0$. Hence, by~\eqref{eq:der-mu}
	\begin{equation}\label{eq:m13}
(\beta_a-\zeta_a^2)\phi_a(0)^2+\phi_a'(0)^2=0.
	\end{equation}
	Consequently, $M_1(a)=0$.

 Now, inserting $p=\frac1{\sigma^2}(\zeta_a+\sigma t)^2$ into \eqref{eq:m11}--\eqref{eq:m13}, we establish~\eqref{eq:m2}.

 A similar computation as above, with the choice $p=\frac{1}{\sigma^2}(\zeta_a+\sigma t)^3$, gives
	\[M_3(a)=\frac 23 \beta_a M_1(a)+\frac 13\Big(\frac 1a-1\Big)\zeta_a\phi_a(0)\phi_a'(0).\]
	Having $M_1(a)=0$, we get~\eqref{eq:m3}.
\end{proof}

\subsection{A model operator in a weighted space}\label{sec:1Dw}\

The operator $\mathfrak h_a[\xi]$ is not sufficient for the understanding of the geometry's influence on the spectrum, as we shall do in Section~\ref{sec:2D-curv} below. For that reason, we introduce a somehow more complicated operator accounting for the curvature term. This is very similar to the setting of the magnetic Neumann Laplacian \cite{helffer2001magnetic}.

We fix $a\in(-1,0)$, $\delta\in(0,\frac1{12})$, $M>0$ and $h_0>0$ such that, for all $h\in(0,h_0)$, $M h^{\frac12-\delta}<\frac13$. 
That way, for  $\mathfrak k\in [-M,M]$, we can introduce the positive function $a_h=(1-\mathfrak k h^\frac12 \tau)$ and the Hilbert space
$L^2\big( (-h^{-\delta},h^{-\delta});a_h\,d\tau\big)$ with the weighted inner product
\[\langle u,v\rangle=\int_{-h^{-\delta}}^{h^{-\delta}} u(\tau)\overline{v(\tau)}\,(1-\mathfrak k h^\frac 12\tau)\,d\tau. \]
For $\xi\in\R$, we  introduce the self-adjoint operator
\begin{multline}\label{eq:H-beta}
\mathcal H_{a,\xi,\mathfrak k,h}=-\frac {d^2}{d\tau^2}+(\sigma \tau+\xi)^2+\mathfrak k h^\frac 12(1- \mathfrak kh^\frac 12\tau)^{-1}\partial_\tau+2\mathfrak k  h^\frac 12 \tau\left(\sigma \tau+\xi-\mathfrak k h^\frac 12 \sigma \frac {\tau^2}2\right)^2\\-\mathfrak k h^\frac 12 \sigma \tau^2 (\sigma \tau+\xi)+\mathfrak k^2 h\sigma^2\frac {\tau^4}4,
\end{multline}
where $\sigma(\cdot)$ is the function in~\eqref{eq:potential}. The domain of definition of this operator is
\begin{equation}\label{eq:dom-1dw}
\dom(\mathcal H_{a,\xi,\mathfrak k,h})=\{u\in H^2(-h^{-\delta},h^{-\delta})~:~u(\pm h^{-\delta})=0\}.
\end{equation}
The operator $\mathcal H_{a,\xi,\mathfrak k,h}$ is the Friedrichs extension in $L^2\big( (-h^{-\delta},h^{-\delta});a_hd\tau\big)$ associated to the quadratic form $q_{a,\xi,\mathfrak k,h}$ defined by
\[q_{a,\xi,\mathfrak k,h}(u)=\int_{-h^{-\delta}}^{h^{-\delta}}\Big(|u'(\tau)|^2+(1+2\mathfrak k h^\frac 12\tau)\Big(\sigma \tau+\xi-\mathfrak k h^\frac 12 \sigma \frac {\tau^2}2\Big)^2u^2(\tau)\Big)(1-\mathfrak k h^\frac 12\tau)\,d\tau.\]
The operator $\mathcal H_{a,\xi,\mathfrak k,h}$ is with compact resolvent. We denote by $\big(\lambda_n(\mathcal H_{a,\xi,\mathfrak k,h})\big)_{n\geq 1}$ its sequence of min-max eigenvalues.

By Theorem~\ref{thm:main}, $\mu_a(\cdot)$ has a unique minimum $\beta_a$ (attained at $\zeta_a$) which is non-degenerate, and the moment $M_3(a)$ in \eqref{eq:m3} is negative, thereby allowing us to derive the following result on the ground state energy of $\mathcal H_{a,\xi,\mathfrak k,h}$.

\begin{proposition}\label{prop:1dw}
Let $\beta_{a,\mathfrak k,h}=\inf\limits_{\xi\in\R}\lambda_1\big(\mathcal H_{a,\xi,\mathfrak k,h}\big)$. Then, as $h\to0_+$,
\[\beta_{a,\mathfrak k,h}=\beta_a+\mathfrak k M_3(a) h^\frac 12+\mathcal O(h^\frac 3{4})\]
uniformly with respect to $\mathfrak k\in[-M,M]$.
\end{proposition}
\begin{proof}
We will present the outline of the proof to show the role  of Theorem~\ref{thm:main}. A similar approach was detailed in~\cite[Theorem 11.1]{helffer2001magnetic} (see also~\cite[Section~4.2]{kachmar2007these}).
By the min-max principle, there exists $C>0$ such that for all $n\geq 1$, $\xi\in\R$ and $h\in(0,h_0)$,
\begin{equation}\label{eq:min-max-1dw}
\big|\lambda_n(\mathcal H_{a,\xi,\mathfrak k,h})-\lambda_n(\mathfrak h_a[\xi])  \big|\leq Ch^{\frac12-2\delta}\big(1+\lambda_n(\mathfrak h_a[\xi]) \big),
\end{equation}
where $\mathfrak h_a[\xi]$ is the fiber operator in \eqref{eq:ham_operator}.  Consequently, we may find a constant $z(a)>0$ such that
\begin{equation}\label{eq:mu1>>beta-a}
{\rm for~}|\xi-\zeta_a|\geq z(a)h^{\frac14-\delta}, ~\lambda_1(\mathcal H_{a,\xi,\mathfrak k,h})\geq \beta_a+  h^{\frac12-2\delta}.
\end{equation}
Note that \eqref{eq:mu1>>beta-a} is a consequence of the fact that $\zeta_a$ is a non-degenerate minimum of $\mu_a(\cdot)$.

Now, we estimate $\lambda_1(\mathcal H_{a,\xi,\mathfrak k,h})$ for $|\xi-\zeta_a|\leq z(a)h^{\frac14-\delta}\ll1$. By \eqref{eq:min-max-1dw}, the simplicity of the eigenvalues $\lambda_n(\mathfrak h_a[\xi]) $ and the continuity of the function $\xi\mapsto\lambda_n(\mathfrak h_a[\xi]) $,  we know that
as $h\to0_+$,
\[ \lambda_1(\mathcal H_{a,\xi,\mathfrak k,h}) =\beta_a+o(1)~\&~\lambda_2(\mathcal H_{a,\xi,\mathfrak k,h})=\lambda_2(\mathfrak h_{a}[\zeta_a])+o(1),\]
with 
\begin{equation}\label{eq:gap}
\lambda_2(\mathfrak h_{a}[\zeta_a])>\lambda_1(\mathfrak h_{a}[\zeta_a])=\beta_a.
\end{equation}
One may construct a formal eigen-pair $(\lambda^{\rm app}_{a,\xi,\mathfrak k,h},f^{\rm app}_{a,\xi,\mathfrak k,h})$ of the operator $\mathcal H_{a,\xi,\mathfrak k,h}$, with
\begin{multline}\label{eq:app-eigenpair}
\lambda^{\rm app}_{a,\xi,\mathfrak k,h}=c_0+c_1(\xi-\zeta_a)+c_2(\xi-\zeta_a)^2+c_3h^{1/2}{\rm ~and}~\\
f^{\rm app}_{a,\xi,\mathfrak k,h}= u_0+(\xi-\zeta_a)u_1+(\xi-\zeta_a)^2u_2+h^{1/2}u_3.
\end{multline}
Expanding $ R_h:=\big(\mathcal H_{a,\xi,\mathfrak k,h}-\lambda^{\rm app}_{a,\xi,\mathfrak k,h}\big)f^{\rm app}_{a,\xi,\mathfrak k,h}$ in powers of $(\xi-\zeta_a)$ and $h^{1/2}$, one can choose $(c_i,u_i)_{0\leq i\leq 3}$ so as the coefficients of the $h^{1/2}$ and $(\xi-\zeta_a)^j$ terms, $j=0,1,2$,   vanish. We choose
\begin{align*}
&c_0=\beta_a,~u_0=\phi_a\\
&c_1=0,~u_1=-2 \mathfrak R_a v_1,~v_1:=(\sigma \tau+\zeta_a)\phi_a\perp \phi_a\\
&c_2=1-4\int_{-\infty}^{+\infty}(\sigma\tau+\zeta_a)\phi_a\mathfrak R_a[(\sigma\tau+\zeta_a)\phi_a]\,dt,~u_2=\mathfrak R_av_2,\\
&\qquad\qquad\qquad
 v_2:=4(\sigma \tau+\zeta_a)\mathfrak R_a[(\sigma\tau+\zeta_a)\phi_a]+(c_2-1)\phi_a\perp\phi_a\\
&c_3=\mathfrak kM_3(a),~
 u_3=\mathfrak R_av_3,\\
 &\qquad\qquad v_3:=-\mathfrak k\Big(\partial_\tau+\frac 1\sigma(\sigma \tau+\zeta_a)^3-\frac {\zeta_a^2}{\sigma}(\sigma\tau+\zeta_a)\Big)\phi_a+c_3\phi_a\perp\phi_a,
\end{align*}
where $\mathfrak R_a\in \mathcal L(L^2(\R))$ is the regularized resolvent introduced in \eqref{eq:R}. That the functions $v_1,v_2,v_3$ are orthogonal to $\phi_a$ is ensured by our choice of $c_1,c_2,c_3$, the expressions of the moments in Proposition~\ref{prop:moment}, and  the first item in Remark~\ref{rem:moment}.

Eventually,  using $\chi( h^{\delta}\tau) f^{\rm app}_{a,\xi,\mathfrak k,h} $ as a quasi-mode, with $\chi$ a cut-off function introduced to insure the Dirichlet condition at $\tau=\pm h^{-\delta}$, we get by the spectral theorem and~\eqref{eq:gap}, 
\begin{equation}\label{eq:min-1Dw-xi}
\lambda_1(\mathcal H_{a,\xi,\mathfrak k,h})= c_0+c_2(\xi-\zeta_a)^2+c_3h^{1/2}+\mathcal O\big(\max(h^{1/2}|\xi-\zeta_a|,|\xi-\zeta_a|^3,h)\big).
\end{equation}
Note that,  for $|\xi-\zeta_a|\leq z(a)h^{\frac14-\delta}$, we have
\[\mathcal O\big(\max(h^{1/2}|\xi-\zeta_a|,|\xi-\zeta_a|^3),h\big)=\mathcal O( h^{3(\frac14-\delta)}).\]
In order to minimize over $\xi$, we observe that the constant $c_2$ can be expressed in the pleasant form\footnote{Using the Feynman-Hellmann formula 
$\mu_a'(\xi)=\langle  (\zeta_a+\sigma(\tau)\tau)\varphi_{a,\xi},\varphi_{a,\xi}\rangle$ \cite[Eq.~(A.9)]{Assaad2019}.}
\[c_2=\frac12\partial^2_\xi \mu_a(\zeta_a),\] 
hence  $c_2>0$ by Theorem~\ref{thm:main}. So, we get from \eqref{eq:mu1>>beta-a} and \eqref{eq:gap},
\begin{equation}\label{eq:min-1Dw}
 \inf_{\xi\in\R}\lambda_1(\mathcal H_{a,\xi,\mathfrak k,h})= c_0+c_3h^{1/2}+\mathcal O(h^{\frac 3{2}(\frac12-\delta)}).
 \end{equation}
 To improve the error in \eqref{eq:min-1Dw},  notice that, by \eqref{eq:min-1Dw-xi}, it is enough to minimize over $\{|\xi-\zeta_a|\leq h^{\frac14}\}$, thereby  finishing the proof of Theorem~\ref{prop:1dw}.
 \end{proof}
\begin{rem}\label{rem:app-ep}
The approximate eigen-pair $(\lambda^{\rm app}_{a,\xi,\mathfrak k,h},f^{\rm app}_{a,\xi,\mathfrak k,h})$ in \eqref{eq:app-eigenpair} does not depend on the parameter  $\delta$ introduced in \eqref{eq:dom-1dw}. Moreover, we have, for $|\xi-\zeta_a|<1$,
\[ \big\| \big(\mathcal H_{a,\xi,\mathfrak k,h} -\lambda^{\rm app}_{a,\xi,\mathfrak k,h}\big) f^{\rm app}_{a,\xi,\mathfrak k,h}\big\|_{L^2(\R)}=\mathcal O\big(\max(h^{1/2}|\xi-\zeta_a|,|\xi-\zeta_a|^3,h)\big) .\]
\end{rem}
\subsection{Magnetic edge \& semi-classical ground state energy}\label{sec:2D-curv}\

With the precise estimate for the ground state energy of  weighted operator of Section~\ref{sec:1Dw} in hand, we can inspect edge states for the Dirichlet Laplace operator with a magnetic step field. 

\subsubsection{Magnetic edge, the domain and the operator}\

Consider a smooth planar curve $\Gamma\subset\R^2$ that splits $\R^2$ into two disjoint unbounded open sets, $P_{\Gamma,1}$ and $P_{\Gamma,2}$. We will refer to $\Gamma$ as the magnetic edge, since we are going to consider magnetic fields having a jump along $\Gamma$ (see Fig.~1).

Now consider an open bounded simply connected  subset $\Omega$ of $\R^2$, with smooth  boundary $\partial\Omega$ of class $C^1$,   and assume that
	\begin{enumerate}
	\item $\Gamma$ intersects $\partial \Om$ at two distinct points $p$ and $q$, and the intersection is transversal, i.e. $\mathrm{T}_{\partial \Omega} \times \mathrm{T}_{\Gamma} \neq 0$ on $\{p,q\}$, where $\mathrm{T}_{\partial \Omega}$ and  $\mathrm{T}_{\Gamma}$ are respectively unit tangent vectors of $\partial \Omega$ and $\Gamma$.
	\item $\Omega_1:=\Omega\cap P_{\Gamma,1}\not=\emptyset$ and  $\Omega_2:=\Omega\cap P_{\Gamma,2}\not=\emptyset$.
	\end{enumerate} 
\begin{figure}
\includegraphics[scale=0.8]{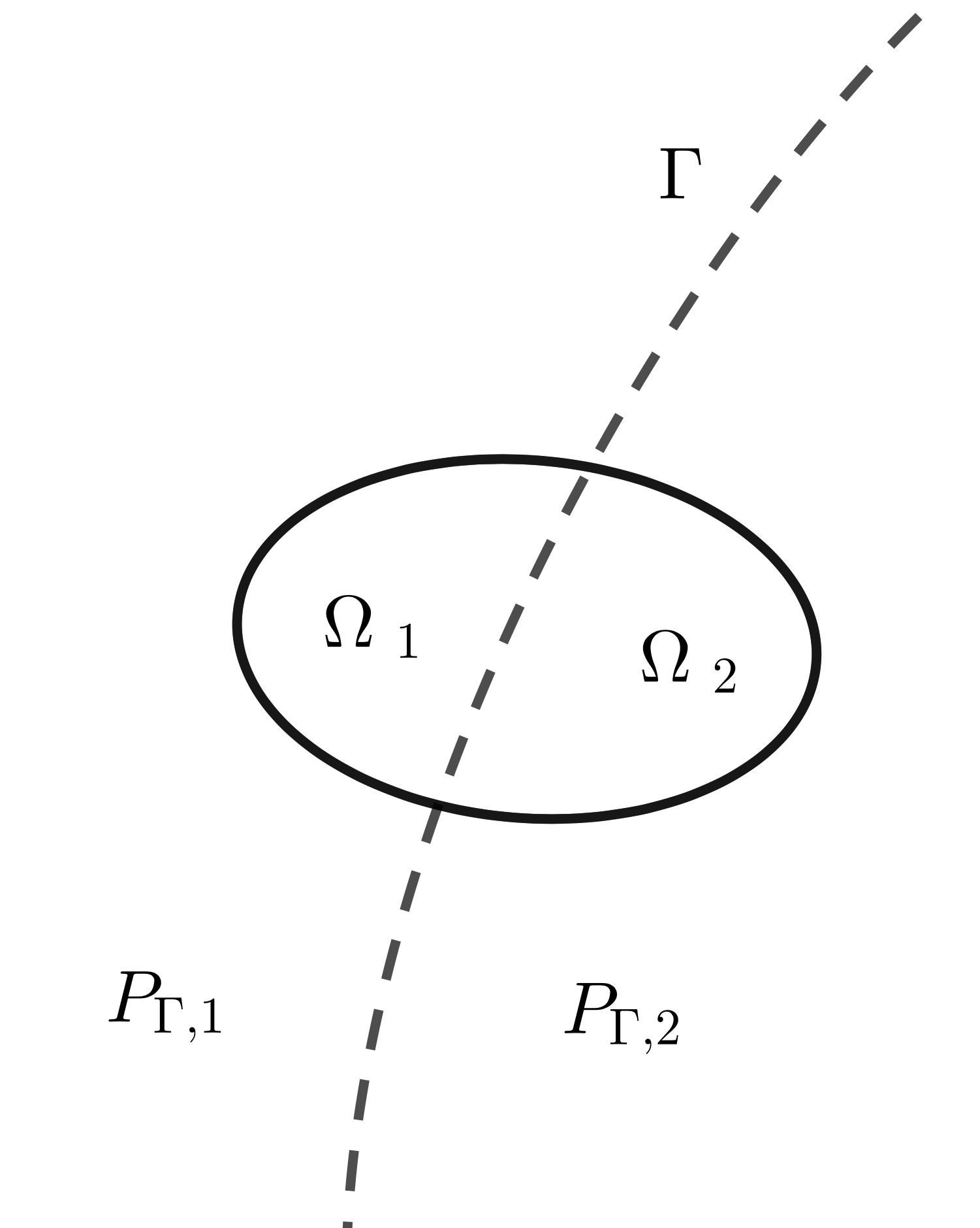}
\caption{The curve $\Gamma$ splits $\R^2$ into two regions, $P_{\Gamma,1}$ \& $P_{\Gamma,2}$, and the domain $\Omega$ into two domains $\Omega_1$ \& $\Omega_2$.}
\end{figure}
Fix $a\in (-1,0)$. Let $\Fb_a \in H^1(\Om,\R^2)$ be a magnetic potential with the corresponding  scalar magnetic field:
\begin{equation}\label{eq:curl-F}
\curl \Fb_a=B_a:=\mathbf 1_{\Omega_1}+a\mathbf 1_{\Omega_2}.
\end{equation}

We consider the Dirichlet realization of the self-adjoint operator in the domain $\Omega$
\[\mathcal P_{h,a}=-(h\nabla-i\Fb_a)^2=-h^2\Delta+ih(\Div \Fb_a+\Fb_a\cdot\nabla)+|\Fb_a|^2,\]
with domain
\[\dom(\mathcal P_{h,a})=\{u\in L^2(\Om)~:~(h\nabla-i\Fb_a)^j u \in L^2(\Om),\, j\in\{1,2\},u|_{\partial \Om}=0 \},\]
and quadratic form
\begin{equation}\label{eq:qf-MLD}
\mathfrak q_{h,a}(u)=\int_\Omega|(h\nabla-i\Fb_a)u|^2\,dx\quad(u\in H^1_0(\Omega)).
\end{equation}
The bottom of the spectrum of this operator is introduced as follows
\begin{equation}\label{eq:gse-ML}
\lambda_1(\mathcal P_{h,a})=\inf_{u\in H^1_0(\Omega)\setminus\{0\}}\frac{\mathfrak q_{h,a}(u)}{\|u\|_{L^2(\Omega)}^2}.
\end{equation}

\subsubsection{Frenet coordinates near the magnetic edge}\label{sec:bc}\

We introduce the Frenet coordinates  near $\Gamma$. We  refer the reader to~\cite[Appendix~F]{fournais2010spectral} and~\cite{Assaad2019} for a similar setup. 

Let $s\mapsto M(s)\in\Gamma$ be  the arc length parametrization of $\Gamma$ such that
\begin{itemize}
\item $\nu(s)$ is the unit normal of $\Gamma$ at the point $M(s)$ pointing to $P_{\Gamma,1}$\,;
\item   $T(s)$ is the unit tangent vector to $\Gamma$ at the point $M(s)$, such that $(T(s),\nu(s))$ is a direct frame, i.e.
$\mathrm{det}\big(T(s),\nu(s) \big)=1$.
\end{itemize}
Now, we define the curvature $k$ of $\Gamma$  as follows
$T'(s)=k(s)\nu(s)$.
For $\epsilon>0$, we define the transformation
\begin{equation}\label{Frenet}
\Phi~:~\R\times(-\epsilon,\epsilon)\, \ni (s,t)\longmapsto M(s)+t \nu(s) \in \Gamma_\epsilon:=\{x\in\R^2~:~{\rm dist}(x,\Gamma)<\epsilon\}.
\end{equation}
We pick $\epsilon$ sufficiently small so that $\Phi$ is a diffeomorphism, whose  Jacobian is 
\begin{equation}\label{eq:a1}
\mathfrak a(s,t):=J_\Phi(s,t)=1-tk(s).
\end{equation}
There is a natural correspondence between functions in  $H^1\big(\Gamma_\epsilon \big)$ and those in $  H^1\big(\R\times(-\epsilon,\epsilon)\big)$. In fact, to every $u\in H^1\big(\Gamma_\epsilon \big)$ we assign $ \tilde u\in  H^1\big(\R\times(-\epsilon,\epsilon)\big)$ 
\begin{equation}\label{eq:u-(s,t)}
\tilde{u}(s,t)=u\big(\Phi(s,t)\big),
\end{equation}
and vice versa.

The vector field $\Fb_a$ can be extended in a natural manner to a vector field in $H^1(\R^2)$. Seen as a  vector field on $\Gamma_\epsilon$, it gives rise to a vector field on $\R\times(-\epsilon,\epsilon)$ as follows
\[\Fb_a(x)=\big(F_{a,1}(x),F_{a,2}(x)\big) \mapsto \tilde{\Fb}_a(s,t)= \big(\tilde F_{a,1}(s,t),\tilde F_{a,2} (s,t)),\]
where 
\begin{equation}\label{eq:A_tild1}
\tilde F_{a,1}(s,t)=\mathfrak a(s,t)\Fb_a\big(\Phi(s,t) \big)\cdot T(s)\quad\mathrm{and}\quad \tilde F_{a,2} (s,t)=\Fb_a\big(\Phi(s,t) \big)\cdot\nu(s).
\end{equation}
Finally, we note the  change of variable formula (for functions compactly supported in $\Gamma_\epsilon$):
\begin{equation}\label{eq:A_tild2}
\begin{aligned}
& \int_{\Gamma_\epsilon}|u |^2\,dx=\int_\R\int_{-\epsilon}^{\epsilon}|\tilde u|^2\, \mathfrak a\,dt\,ds~\&~\\
&\int_{\Gamma_\epsilon}\big|\big(h\nabla-i \Fb_a \big)u \big|^2\,dx=\int_\R\int_{-\epsilon}^{\epsilon}\left(\mathfrak a^{-2}\big|(h\partial_s-i\tilde{F}_{a,1})\tilde{u}\big|^2+\big|(h\partial_t-i\tilde{F}_{a,2})\tilde{u}\big|^2 \right)\, \mathfrak a\,dt\,ds.
\end{aligned}
\end{equation}

\subsubsection{Ground state energy and curvature of the magnetic edge}\

We introduce the maximal curvature of $\Gamma$ in $\Omega$ as follows
\begin{equation}\label{eq:k-max}
k_{\max}^\Omega=\max_{x\in\Gamma\cap\overline{\Omega}}\Big(k\big(\Phi^{-1}(x)\big)\Big).
\end{equation}

\begin{theorem}\label{thm:D-Lap}
There exist positive constants $c_a,C_a,h_a$ such that the ground state energy in \eqref{eq:gse-ML} satisfies, for all $h\in(0,h_a)$,
\[-c_ah^\frac {5}{3}\leq \lambda_1(\mathcal P_{h,a})-\big(\beta_a h+M_3(a) k_{\max }^\Omega h^{\frac 32}\big)\leq C_ah^\frac {7}{4}.\]
\end{theorem}

\subsubsection{Upper bound on the ground state energy}\

This will be done by the construction of a trial state involving an appropriate gauge transformation in the Frenet coordinates that we recall below.

\begin{lemma}\label{lem:Anew2}
For $x_0=\Phi(s_0,0)\in\Gamma$ and $0<\ell<\epsilon$, we introduce the neighborhood $\mathcal N(x_0,\ell)=\{\Phi(s,t)~: |s-s_0|<\ell~\&~|t|<\ell\}$.  There exists a function $\omega_\ell
	\in \mathcal N(x_0,\ell)$  such that the vector potential $\tilde{\Fb}^{\sf new}_{a}:=\tilde{\Fb}_a-\nabla_{s,t}\omega_\ell$, defined on $\mathcal N(x_0,\ell)$, satisfies
	\begin{equation}\label{eq:Anew3}
	\tilde{F}_{a,1}^{\sf new}(s,t)= 
	\begin{cases}
	- \big(t-\frac {t^2}2 k(s)\big)&\mathrm{if}~t>0\medskip\\
	-a \big(t-\frac {t^2}2 k(s)\big)&\mathrm{if}~t<0
	\end{cases}
	\quad\&\quad \tilde{F}^{\sf new}_{a,2}(s,t)=0.
	\end{equation}
\end{lemma}

Now pick $x_0=\Phi(s_0,0)\in \Gamma\cap\overline\Omega$ such that $k(s_0)=\kappa_{\max}^\Omega$. Select $x_h=\Phi(s_h,0)\in \Gamma\cap\Omega$ so that $|s-s_h|=h^{1/8}$. We introduce the trial state $u$ defined in the Frenet coordinates as follows
\begin{multline}\label{eq:trial-state-2D}
u(\Phi(s,t))=\tilde u(s,t)\\
=c_h \chi\Big(\frac{s-s_h}{h^{1/8}}\Big) \chi\Big(\frac{t}{h^{1/6}}\Big)f^{\rm app}_{a,\zeta_a,k(s_0),h}(h^{-1/2}t) \exp\Big(\frac{i\big(\zeta_a s-\omega(s,t)\big)}{h^{1/2}}\Big),
\end{multline}
where $\omega=\omega_\ell$ is the gauge function introduced in Lemma~\ref{lem:Anew2} for $\ell=2h^{1/6}$, $f^{\rm app}_{a,\zeta_a,k(s_0),h}$ is the approximate 1D eigenfunction introduced in \eqref{eq:app-eigenpair} with $\xi=\zeta_a$, $\chi$ is a cut-off function and $c_h>0$ is a constant selected so that the $L^2$-norm of $u$ in $\Omega$  is equal to $1$. We choose the cut-off function as follows:
\[\chi\in C_c^\infty(\R),~{\rm supp}\,\chi\subset[-1,1],~\chi=1~{\rm on~}[-1/2,1/2].\]
Then, we can compute $\mathfrak q_{h,a}(u)$ and get
\[\lambda_1(\mathcal P_{h,a})\leq \frac{\mathfrak q_{h,a}(u)}{\|u\|_{L^2(\Omega)}}\leq \beta_a h+k(s_0)M_3(a)h^{3/2}+\mathcal O(h^{7/4}).\]

\subsubsection{Concentration near the magnetic edge}

Fix $R_0>1$ and consider a partition of unity 
\[\sum_{j=1}^{N_h} \chi_{h,j}^2=1~{\rm in }~\Gamma_{R_0h^{1/2}}\]
such that
\[{\supp}\,\chi_{h,j}\subset \mathcal N(x_j,R_0h^{1/2})~\&~\sum_{j=1}^{N_h} |\nabla\chi_{j,h}|^2=\mathcal O(R_0^{-2}h^{-1}).\]
Also, we assume that $x_1=p$, $x_{N_h}=q$, where $\{p,q\}=\Gamma\cap\partial\Omega$.

We introduce another partition of unity $\sum\limits_{i=1}^2\varphi_{i,h}^2=1$ in $\R^2$ such that ${\supp}\varphi_{1,h}\subset\R^2\setminus\Gamma_{R_0h^{1/2}}$ and $\sum\limits_{i=1}^2|\nabla\varphi_{i,h}|^2=\mathcal O(R_0^{-2}h^{-1})$.

Pick an arbitrary $u\in H^1_0(\Omega)$. We extend $u$ by $0$ on $\R^2\setminus\Omega$. Notice that
\begin{align*}
\mathfrak q_{h,a}(u) &=\sum_{i=1}^2 \mathfrak q_{h,a}(\varphi_{i,h}u)-h^2\sum_{i=1}^2 \big\| |\nabla\varphi_{i,h}|u\big\|^2\\
&= \mathfrak q_{h,a}(\varphi_{1,h}u) +\sum_{j=1}^{N_h} \Big(\mathfrak q_{h,a}\big(\varphi_{2,h}\chi_{j,h}u\big)-h^2 \big\| |\nabla \chi_{j,h}|\varphi_{2,h}u\big\|^2\Big)-h^2\sum_{i=1}^2 \big\| |\nabla\varphi_{i,h}|u\big\|^2\\
&=\mathfrak q_{h,a}(\varphi_{1,h}u) +\sum_{j=1}^{N_h} \mathfrak q_{h,a}\big(\varphi_{2,h}\chi_{j,h}u\big)-\mathcal O(R_0^{-2}h).
\end{align*}
We bound from below each $ \mathfrak q_{h,a}\big(\varphi_{2,h}\chi_{j,h}u\big)$ as follows (see \cite{Assaad2020}) \\
\[  \mathfrak q_{h,a}\big(\varphi_{2,h}\chi_{j,h}u\big)\geq \big(\beta_a h -\mathcal O (h^{3/2} ) \big)\big\|\varphi_{2,h}\chi_{j,h}u \big\|^2.\]
Since $\curl\Fb_a$ is constant away from $\Gamma$, we bound $\mathfrak q_{h,a}(\varphi_{1,h}u) $ from below as follows
\[\mathfrak q_{h,a}(\varphi_{1,h}u) \geq \int_{\Omega}|\curl\Fb_a| \,|\varphi_{1,h}u|^2\,dx\geq |a|h\big\| \varphi_{1,h}u\big\|^2.   \]
Summing up, we deduce the following  lower bound  on the quadratic form
\[ \mathfrak q_{h,a }(u)\geq \int_\Omega \big(U_{h,a}(x)-\mathcal O(R_0^{-2}h)\big)|u(x)|^2\,dx\quad(u\in H^1_0(\Omega)),\]
where 
\[U_{h,a}(x)=\begin{cases}
|a|h &{\rm if~}{\rm dist}(x,\Gamma)>R_0 h^{1/2} \\
\beta_a h&{\rm if~}{\rm dist}(x,\Gamma)<R_0 h^{1/2}
\end{cases}.
\]
This allows  us to do Agmon estimates and arrive at the following decay property of eigenfunctions $u_h$ with eigenvalues $z_h\leq \beta_a h+o(h)$:
\begin{equation}\label{eq:dec-gs}
 \int_\Omega \Big(|u_h|^2+h^{-1}|(h\nabla-i\Fb_{a})u_h|^2 \Big)\exp\Big(\frac{\alpha\,{\rm dist}(x,\Gamma)}{h^{1/2}} \Big)\,dx\leq C\|u_h\|_{L^2(\Omega)},
 \end{equation}
 for some positive  constants $\alpha$ and $C$.

As a consequence of \eqref{eq:dec-gs} (and the inequality $e^z\geq \frac{z^n}{n!}$ for $z\geq 0$), we get for any positive integer $n$,
\begin{equation}\label{eq:dec-gs*}
\int_\Omega \big({\rm dist}(x,\Gamma)\big)^n\Big(|u_h|^2+h^{-1}|(h\nabla -i\Fb_{a})u_h|^2\Big)\,dx\leq C_nh^{n/2}\|u_h\|^2_{L^2(\Omega)},
\end{equation}
for a positive constant $C_n$.

\subsubsection{Lower bound on the ground state energy}

Pick a ground state $u_h$ of $\lambda_1(\mathcal P_{h,a})$ and extend it by $0$ on $\R^2\setminus\Omega$. We will bound the quadratic form from below as follows
\begin{equation}\label{eq:qf-uh-lb}
\mathfrak q_{h,a}(u_h)\geq \Big(\beta_a h+M_3(a) k_{\max }^\Omega h^{\frac 32}-\mathcal O(h^{\frac {5}3})\Big)\|u_h\|^2_{L^2(\Omega)}.
\end{equation}
Set $\epsilon_h=h^{\frac12-\delta}$, with $\delta\in(0,\frac1{12})$. Consider two partitions of unity
\[ \sum\limits_{i=1}^2\varphi_{i,h}^2=1~{\rm  in~} \R^2,~ {\supp}\varphi_{1,h}\subset\R^2\setminus\Gamma_{\epsilon_h},~\sum\limits_{i=1}^2|\nabla\varphi_{i,h}|^2=\mathcal O(h^{2\delta-1}),\] 
and, for a fixed $\rho\in(0,\frac12)$,
\[ \sum_{j=1}^{N_h} \chi_{h,j}^2=1~{\rm in }~\Gamma_{h^{\rho}}, ~{\supp}\,\chi_{h,j}\subset \mathcal N(x_j,h^{\rho}),~\sum_{j=1}^{N_h} |\nabla\chi_{j,h}|^2=\mathcal O(h^{-2\rho}),\]
with $x_j=\Phi(s_j,0)\in\Gamma\cap\overline{\Omega}$, $x_1=p$, $x_2=q$ and $\{p,q\}=\Gamma\cap\partial\Omega$. Set $w_h=\varphi_{2,h}u_h$. By \eqref{eq:dec-gs}
\begin{equation}\label{eq:wh=uh}
\|w_h\|_{L^2(\Omega)}=\|u_h\|_{L^2(\Omega)}+\mathcal O(h^\infty)~\&~\mathfrak q_{h,a}(w_h)=\mathfrak q_{h,a}(u_h)+\mathcal O(h^\infty).
\end{equation}
Now, we  decompose $\mathfrak q_{h,a}(w_h)$ via  the  partition of unity along $\Gamma$ as follows 
\begin{equation}\label{eq:q-ha-w}
 \mathfrak q_{h,a}(w_h)=\sum_{j=1}^{N_h} \mathfrak q_{h,a}(w_{h,j})+\mathcal O(h^{2-2\rho})\|w_h\|^2_{L^2(\Omega)}~{\rm with~}w_{h,j}=\chi_{h,j}\varphi_{2,h}u_h.\end{equation}
Performing  a local gauge transformation in $\mathcal N(x_j,h^{\rho})$ as in Lemma~\ref{lem:Anew2}, we get a new function $\tilde w_{h,j}$ such that
\[ 
\mathfrak q_{h,a}(w_{h,j})
=\int_\R\int_{-\epsilon_h}^{\epsilon_h} 
\left(\mathfrak a^{-2}\Big|\Big(h\partial_s+i\sigma t-\frac{\sigma t^2}{2}k(s)\Big)\tilde{w}_{h,j}\Big|^2
+h^2|\partial_t\tilde{w}_{h,j}|^2 \right)\, \mathfrak a\,dt\,ds 
\]
In every $\mathcal N(x_j,h^\rho)$, we expand  
\[\kappa(s)=\kappa_j+\mathcal O(h^{\rho}),~\mathfrak a=1-t\kappa_j+\mathcal O(h^{\rho}t),~\mathfrak a^{-2}=1+2\kappa_j t+\mathcal O(h^{\rho}t),\]
where, 
\begin{equation}\label{eq:min-kappa}
\kappa_j:=\kappa(\bar s_j)=\min_{|s-s_j|\leq h^{\rho}}\kappa(s), ~x_j=\Phi(s_j,0)~\&~\bar s_j\in\{|s-s_j|\leq h^{\rho}\}\,.
\end{equation}
For every integer $n\geq 0$, we write by \eqref{eq:dec-gs*},
\begin{align*}
&\sum_{j=1}^{N_h}\int_\R\int_{-\epsilon_h}^{\epsilon_h}  |t|^n| \tilde w_{h,j}|^2dt\,ds\leq \tilde C_n h^{\frac{n}2}\|u_h\|^2_{L^2(\Omega)}\\
&\sum_{j=1}^{N_h}\int_\R\int_{-\epsilon_h}^{\epsilon_h}  h^2|t|^n|\partial_t \tilde w_{h,j}|^2dt\,ds\leq \tilde C_n h^{1+\frac{n}2}\|u_h\|^2_{L^2(\Omega)}\\
&\sum_{j=1}^{N_h}\int_\R\int_{-\epsilon_h}^{\epsilon_h}  |t|^n\Big|\Big(h\partial_s+i\sigma t-\frac{\sigma t^2}{2}k(s)\Big)\tilde{w}_{h,j}\Big|^2dt\,ds\leq {\tilde C_n} h^{1+\frac{n}2}\|u_h\|^2_{L^2(\Omega)}.
\end{align*}
That way we get
\begin{multline}\label{eq:sum-q-h,a-wj*}
 \sum_{j=1}^{N_h} \mathfrak q_{h,a}(w_{h,j}) \geq
 \sum_{j=1}^{N_h}\int_\R\int_{-\epsilon_h}^{\epsilon_h}\Big((1+2\kappa_j)t)\Big|\Big(h\partial_s+i\sigma t-\frac{\sigma t^2}{2}k_j\Big)\tilde{w}_{h,j}\Big|^2\\
+h^2|\partial_t\tilde{w}_{h,j}|^2 \Big)\, (1-t\kappa_j)\,dt\,ds  -\mathcal O(h^{\frac32+\rho})
\end{multline}
In each $\{|s-s_j|<h^{\rho}\}\cap\{|t|<h^{\frac12-\delta}\}$, we perform a partial Fourier transform w.r.t. $s$ and the scaling $t\mapsto \tau=h^{-\frac12}t$. We then reduce to the setting of Proposition~\ref{prop:1dw} and get, after summing over $j$,
\begin{equation}\label{eq:sum-q-h,a-wj}
\sum_{j=1}^{N_h} \mathfrak q_{h,a}(w_{h,j})\geq h\sum_{j=1}^{N_h}\Big( \beta_a+M_3(a)\kappa_jh^{\frac12}+\mathcal O(h^{\frac{3}{4}}) -\mathcal O(h^{\frac12+\rho})\Big)\|w_{h,j}\|^2_{L^2(\Omega)}.
\end{equation}
Noticing  that $\sum\limits_{j=1}^{N_h}\|w_{h,j}\|^2_{L^2(\Omega)}=\|w_h\|^2_{L^2(\Omega)}$, the following holds
\begin{multline*}
\sum_{j=1}^{N_h} \mathfrak q_{h,a}(w_{h,j})\geq\\
  h\int_\R\int_{-\epsilon_h}^{\epsilon_h}\Big( \beta_a+M_3(a)\kappa(s)h^{\frac12}-\mathcal O(h^{\frac{3}{4}}) -\mathcal O(h^{\frac12+\rho})\Big)|\tilde w_{h}|^2(1-t\kappa(s))dtds
\end{multline*}
since $M_3(a)<0$, by Proposition~\ref{prop:moment}, and $\kappa_j\leq \kappa(s)$ in the support of $w_{h,j}$, by \eqref{eq:min-kappa}.
 Inserting this into \eqref{eq:q-ha-w}, we get
\[q_{h,a}(w_h)\geq \int_\R\int_{-\epsilon_h}^{\epsilon_h}\Big( \beta_a h+M_3(a)\kappa(s)h^{\frac32}-\mathcal O\big(\max(h^{\frac{7}{4}},h^{\frac32+\rho}, h^{2-2\rho})\big)\Big)|\tilde w_{h}|^2dtds\,.
\]
Now, by \eqref{eq:wh=uh}, we get 
\[ \lambda_1(\mathcal P_{h,a})\geq \beta_ah +M_3(a)\kappa_{\max}h^{\frac32}-\mathcal O \big(\max(h^{\frac{7}{4}},h^{\frac32+\rho}, h^{2-2\rho})\big)\,.\]
Optimizing, we choose $\rho=\frac16$ and get that the remainder is $\mathcal O(h^{\frac{5}3})$.

\begin{rem}\label{rem:gs-conc}
Let us introduce the potential 
\[ U^\Gamma_{h,a}(x) =\begin{cases}
|a|h&{\rm if~}{\rm dist}(x,\Gamma)>2h^{\frac16}\\
\beta_ah+M_3(a)\kappa(s)h^{\frac32}&{\rm if}~{\rm dist}(x,\Gamma)<2h^{\frac16}~\&~x=\Phi(s,t)
\end{cases}.\]
Then, repeating the foregoing proof  (with  $\rho=\frac16$) on the Schr\"odinger operator
\[\mathcal P_{h,a}-U^\Gamma_{h,a},\]
we get that its ground state energy satisfies
\[\lambda(h,a,\Gamma)\geq - \mathfrak c h^{\frac{5}3} \]
for some positive  constant $\mathfrak c$. Therefore, we deduce that, for any $u\in H^1_0(\Omega)$, the following inequality holds
\begin{equation}\label{eq:qh-u-lb}
\mathfrak q_{h,a}(u)\geq \int_\Omega\big(U_{h,a}^\Gamma(x)-\mathfrak c  h^{\frac{5}3}\big)\big)|u|^2\,dx.
\end{equation}
The inequality in \eqref{eq:qh-u-lb} yields that the ground states of $\mathcal P_{h,a}$ are localized near the set of maximal magnetic edge curvature, $\Pi_\Gamma=\{\kappa(s)=\kappa_{\max}^\Omega\}$. We omit the details and refer the reader to \cite[Thm.~8.3.4]{fournais2010spectral}. 
\end{rem}

\subsection{Superconductivity along the magnetic edge}\label{sec:GL}\

The new estimate $\beta_a<\Theta_0$ in Theorem~\ref{thm:main} gives an integrated description of the nucleation of superconductivity in type-II superconductors subject to magnetic steps fields with certain intensity, considered for instance in~\cite{Assaad2019}.

In the context of superconductivity,  the set $\Om$ introduced in~Section~\ref{sec:2D-curv} models the horizontal cross section of a cylindrical superconductor-sample, with a large characteristic parameter $\kappa$ and submitted  to the magnetic field $HB_a$, where $B_a$ is as in \eqref{eq:curl-F},  $a\in(-1,0)$, and  the parameter $H>0$ measures the intensity of the magnetic field. The superconducting properties of the sample are described by the minimizing configurations of the following Ginzburg--Landau (GL) energy functional:
 \begin{equation}\label{eq:GL}
 \GL (\psi,\Ab)= \int_\Om \Big( \big|(\nb-i\kp H {\mathbf
 	A})\psi\big|^2-\kp^2|\psi|^2+\frac{\kappa^2}{2}|\psi|^4 \Big)\,dx
 +\kp^2H^2\int_\Om\big|\curl\Ab-B_a\big|^2\,dx,
 \end{equation}
where $\psi \in H^1(\Om;\C)$ is the order parameter, and  $\Ab\in H^1(\Om;\R^2)$ is the induced magnetic field. For a fixed $(\kappa,H)$, the infimum of the energy--the  ground state energy--is attained by a minimizer $(\psi^{\mathrm {GL}},\Ab^{\mathrm {GL}})_{\kappa,H}$.

In  \cite{Assaad2019}, the limit profile of  $|\psi^{\rm GL}|^4$ is determined in the sense of distributions in the regime where $H=b\kappa$ and $\kappa\to+\infty$, with $b>\frac1{|a|}$ a fixed constant. More precisely, the following convergence holds
 
 \[\kappa\mathcal T_\kappa^b  \rightharpoonup \mathcal T^b\ \mathrm{in}\ \mathcal D'(\R^2),\ \mathrm{as}\ \kappa\rightarrow+\infty,\]
 where 	 
 \[C_c^\infty(\R^2)\ni \varphi\mapsto\mathcal T_\kappa^b(\varphi)=\int_\Omega |\psi^{\rm GL}|^4 \varphi\,dx\]
 	and the limit distribution $\mathcal T^b$ is defined via three distributions related to the edges $\Gamma$, $\Gamma_1=(\partial\Omega_1)\cap(\partial\Omega)$ and $\Gamma_2=(\partial\Omega_2)\cap(\partial\Omega)$ as follows
\[
C_c^\infty(\R^2) \ni \varphi \mapsto \mathcal T^b(\varphi)=-2b^{-\frac 12}
\big(\mathcal T^b_{\Gamma}(\varphi)+\mathcal T^b_{\Gamma_1}(\varphi)+ \mathcal T^b_{\Gamma_2}(\varphi)\big),\]
with
\begin{multline*}
\mathcal T^b_{\Gamma}(\varphi):=\mathfrak e_a(b)\int_{\Gamma}\varphi \,ds_{\Gamma},
\quad \mathcal T^b_{\Gamma_1}(\varphi)= E_\mathrm{surf}(b)\int_{\Gamma_1}\varphi \,ds ~\&\\
\mathcal T^b_{\Gamma_2}(\varphi\big)=|a|^{-\frac 12}E_\mathrm{surf}\big(b|a|\big)\int_{\Gamma_2}\varphi \,ds.
\end{multline*}
The effective energies $\mathfrak e_a$ and $E_\mathrm{surf}$  correspond respectively to the contribution of the magnetic edge $\Gamma$ and the boundary $\partial \Om$ (see \cite{Assaad2019, Correggi} for the precise definitions). They have  the following properties: 
\begin{itemize}
	\item $\mathfrak e_a(b)=0$ if and only if $b\geq 1/\beta_a$.
		\item $E_\mathrm{surf}(b)=0$ if and only if $b\geq 1/\Theta_0$.
\end{itemize}
Based on the results above, a detailed discussion on the distribution of superconductivity near $\Gamma\cup\partial \Om$ has been done in~\cite[Section~1.5]{Assaad2019}. This discussion mainly relies on the order of the values $|a|\Theta_0$, $\beta_a$ and $\Theta_0$. With the existing estimates in this paper (and \cite{Assaad2019}), we have
\[|a|\Theta_0<\beta_a<\min(\Theta_0,|a|)~{\rm for}~a\in(-1,0).\]
Consequently, we observe that (see Fig~\ref{fig1} for illustration)

\begin{figure}
	\begin{subfigure}{.3\linewidth}
		\centering
	\includegraphics[scale=0.5]{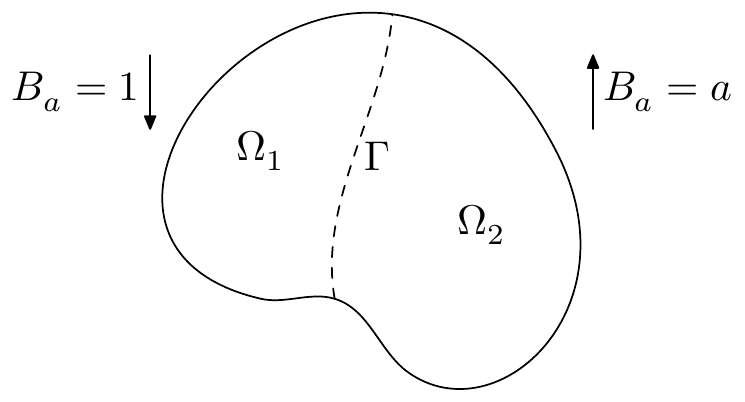}
	\end{subfigure}%
	\begin{subfigure}{.3\linewidth}
		\centering
		\includegraphics[scale=0.5]{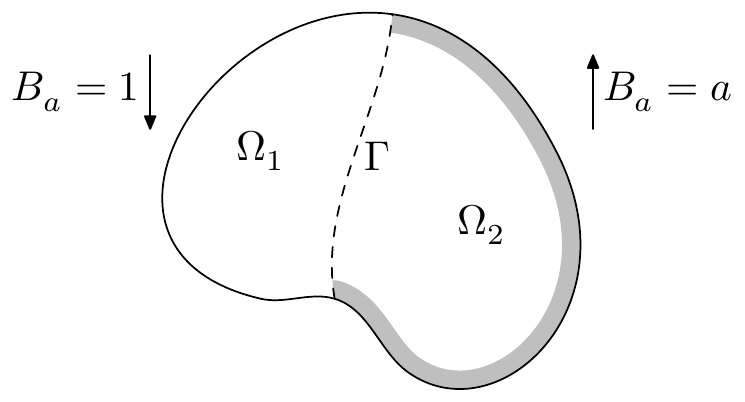}
	\end{subfigure}%
	\begin{subfigure}{.3\linewidth}
		\centering
		\includegraphics[scale=0.5]{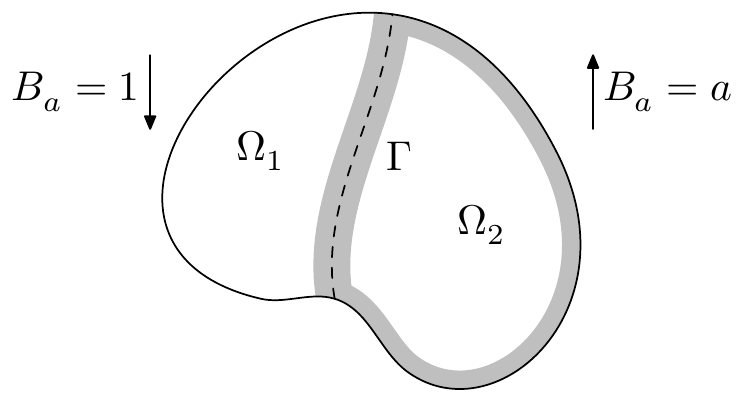}
	\end{subfigure}
	\caption{Superconductivity localization in the set $\Omega$ submitted to the magnetic field $B_a$, for $a\in (-1,0)$, with intensity $H=b\kappa$, where respectively $b\geq b_{c,3}:=\frac1{|a|\Theta_0}$, $b_{c,2}:=\frac1{\beta_a} \leq b<b_{c,3}$ and $b_{c,1}:=\max(\frac1{|a|},\frac1{\Theta_0})\leq b<b_{c,2}$. Only the grey regions carry superconductivity.}
	\label{fig1}
\end{figure}
\begin{itemize}
\item $\mathcal T^b=0$ for $b\geq b_{c,3}:=\frac1{|a|\Theta_0}$\,;\medskip
\item $\mathcal T^b_{\Gamma_1}=\mathcal T^b_{\Gamma}=0$ \& $\mathcal T^b_{\Gamma_2}\not=0$ for 
$b_{c,2}:=\frac1{\beta_a} \leq b<b_{c,3}$\,;\medskip
\item $\mathcal T^b_{\Gamma_1}=0$, $\mathcal T^b_{\Gamma_1}\not=0$ and $\mathcal T^b_{\Gamma_2}\not=0$ for 
$b_{c,1}:=\max(\frac1{|a|},\frac1{\Theta_0})\leq b<b_{c,2}$.
\end{itemize}

\end{document}